\documentclass[a4paper, 10pt, conference]{ieeeconf}      

\usepackage[applemac]{inputenc}
\usepackage{graphicx} 

\usepackage{graphics} 
\usepackage{epsfig} 
\usepackage{amsmath} 
\usepackage{amssymb}  
\usepackage{amsmath} 
\usepackage{amssymb}  
\usepackage{cite}
\usepackage{fullpage}
\usepackage{fancyhdr}
\usepackage{color}
\usepackage{verbatim}
\usepackage{color}
\usepackage{graphicx}      
\usepackage{psfrag}

\newtheorem{theorem}{Theorem}
\newtheorem{corollary}{Corollary}

\newtheorem{lemma}{Lemma}

\newtheorem{proposition}{Proposition}

\newtheorem{remm}{Remark}
\newenvironment{remark}{\begin{remm}\rm }{\hfill \hspace*{1pt} \hfill
$\star$\end{remm}}

\newcommand{\nt}{{\mathbb N}}

\newcommand{\real}{{\mathbb R}}
\newcommand{\R}{{\mathbb R}}

\newcommand{\Ltwo}{{\mathcal{L}^2}}
\newcommand{\Ck}{{\mathcal{C}^k}}

\title{\LARGE \bf
Semi-definite programming and functional inequalities for Distributed Parameter Systems
}

\author{ G. Valmorbida\and M. Ahmadi\and A. Papachristodoulou
\thanks{Department of Engineering Science, University of Oxford,  17 Parks Road, OX1 3PJ Oxford, United Kingdom,  Email: \{giorgio.valmorbida,~mohamadreza.ahmadi,~antonis\}@eng.ox.ac.uk. G. Valmorbida is also affiliated to Somerville College, University of Oxford, Oxford, U.K. Work supported by EPSRC grant EP/J010537/1. A. Papachristodoulou was supported in part by the Engineering and Physical Sciences Research Council projects EP/J012041/1, EP/I031944/1 and EP/J010537/1. M. Ahmadi is supported by Oxford Clarendon Scholarship and Keble College Sloane-Robinson Scholarship.
} 
}

\IEEEoverridecommandlockouts
\begin{document}

\maketitle
\thispagestyle{empty}
\pagestyle{empty}

\begin{abstract}
We study one-dimensional integral inequalities, with quadratic integrands, on bounded domains. Conditions for these inequalities to hold are formulated in terms of function matrix inequalities which must hold in the domain of integration. For the case of polynomial function matrices, sufficient conditions for positivity of the matrix inequality and, therefore, for the integral inequalities are cast as semi-definite programs. The inequalities are used to study stability of linear partial differential equations. 

\end{abstract}

\section{Introduction}
\label{sec:intro}
Emerging applications~\cite{BPWB13,CWPD13, VSK08,VZ09,CA98} (Magnetohydrodynamics, fluids, population dynamics) and stringent performance requirements have recently driven control engineering researchers interest towards systems described by partial differential equations (PDEs), that is, equations involving derivatives with respect to more than a single independent variable. Usually the set of independent variables are time and spatial variables, and the solution to the PED solution is assumed to be forward complete, meaning that the domain is unbounded for the temporal variable. On the other hand, solutions to equations representing physical systems are often defined in a bounded \textit{spatial domain}. 

Several numerical approaches for the analysis and control design of PDE systems rely on ODEs, obtained by spectral truncation or spatial discretization, approximating the PDE model with a finite number of states~\cite{GC12},~\cite{FAC03}.  As for ODEs,  conditions for stability of the zero solution  can be formulated from spectral analysis when the PDE system is defined by a linear operator. Moreover it is possible to infer stability from the semi-group generated by linear or nonlinear operators and this parallels the ODE approach of obtaining a solution to establish stability of a particular solution~\cite{CZ95}. An alternative approach is to rely on the Lyapunov method, extended to infinite dimensional systems in~\cite{Mov59} and~\cite{Dat70}, which does not require the semi-groups to be calculated. The energy of the state, which for PDEs takes values in a function space instead of an Euclidean one, is a frequent choice for the Lyapunov functional (LF) since it simplifies the analysis of a large class of  nonlinear PDE systems whenever the nonlinearities are energy-preserving~\cite{Str04}. However, using fixed LFs may be conservative and is preferable to consider a family of parameterised functionals. The choice for the class of parameterised functionals should be supported by a Lyapunov converse theorem.  

Even for one-dimensional spatial domain PDEs, the current development of Lyapunov analysis rely on analytical steps~\cite{Str04}. These steps present increasing complexity for systems of several dependent variables, for systems with spatially varying properties (anisotropic systems) and for LF integrands depending on the spatial variable. 

Semi-definite programming (SDP) has recently been successfully applied to control problems with polynomial data being formulated as convex optimization problems. Among those, one can cite stability of time-delay systems~\cite{PPL07}, synthesis  of polynomial control  laws~\cite{VTG13}~\cite{PPW04}, robustness analysis of polynomial systems~\cite{TPSB10} giving SOS programs (SOSP), while the primal formulation of the SOSP, the generalised problem of moments~\cite{Las09}, has been applied to optimal control problems~\cite{LHPT08} and system analysis~\cite{HK14}.

While the connection of polynomial inequalities to semi-definite constraints was possible thanks to the non-uniqueness of quadratic-like representation of polynomials (parametrised by Gram matrices~\cite{CLR95}) the non-uniqueness of integral expressions with integrands being quadratic expressions on the dependent variables has not yet been explored. A hint on this direction for integral operators was reported in~\cite{PP06}, where the use of integration by parts associated to Dirichlet boundary condition was instrumental to formulate the stability test for a PDE with dissipation and reaction terms as an SOSP. 

With the purpose of formulating numerical tests for the analysis of PDE systems, this paper studies one-dimensional integral inequalities whose integrands are functions of the independent spatial variables, of the dependent variables and their spatial derivatives. The integrand is assumed to be \emph{quadratic} on the dependent variable and \emph{polynomial} on the spatial variable.

The fundamental theorem of calculus (FTC) is the key step to relate the dependent variables and their derivatives in an integral expression. This step allows us to obtain a set of quadratic expressions which do not affect the positivity of the integral. The matrices on these quadratic expressions depend on the spatial variables and their entries relate to the values the dependent variables assume on the boundaries. The positivity check of the integral on the domain is then performed by a check of the positivity of the matrix inequalities, involving the quadratic expression on the original inequality and the ones obtained with the FTC. For polynomial matrices on the independent variables, we rely on the Positivstellensatz~\cite{Put93} in order to generate SOS programs yielding, therefore, a problem to be solved numerically. 


The above results are then applied to study the stability of the $\mathcal{L}_2$ norm of systems of anisotropic PDEs with weighted $\mathcal{L}_2$ norm as LF candidates. Several numerical examples illustrate the results: bounds for the Poincar\'e inequalities are derived numerically, the stability of the heat equation with spatially varying coefficients is studied, the transport equation, and a set of reaction-diffusion equation~\cite{Str04}, leading to integral inequalities whose integrand is a quadratic expressions on the dependent variable. 



\textit{Notation} Let $\R, \R_{\ge 0}, \R_{>0}$ and $\R^{n}$ denote the field of reals, non-negative reals, positive reals and the $n$-dimensional Euclidean space respectively. The sets of natural numbers and positive natural numbers are denoted $\nt^n$, $\nt^n_0$. The closure of set $\Omega$ is denoted~$\overline{\Omega}$. The boundary $\partial \Omega$ of set $\Omega$ is defined as $\overline{\Omega} \setminus \Omega$ with ``$\setminus$'' denoting set substraction. The ring  of polynomials, the ring of positive polynomials, and the ring of sum-of-squares polynomials on real variable $x \in \real$ are respectively denoted $\mathcal{R}[x]$, $\mathcal{P}[x]$ and $\Sigma[x]$. The ring of Sum-of-squares matrices of dimensions $n$ is denoted $\Sigma^{n \times n}[x]$. The set of functions in a Hilbert space $H$ on $\Omega$ are denoted $H(\Omega)$. We denote the the space of measurable functions defined on $\Omega$ as $u \in \Ltwo_{\Omega}$ we denote the spatial $\Ltwo_\Omega$-norm by $\|u(t)\|_{2,\Omega} = \left(\int_{\Omega} u^{T}(t,x)u(t,x) dx \right)^{\frac{1}{2}}$ we use $\Ltwo_{P,\Omega}$ to denote the weighted $\Ltwo$ norm $\|u(t)\|_{(2,P),\Omega} = \left(\int_{\Omega} u^{T}(t,x)P(x)u(t,x) dx \right)^{\frac{1}{2}}$. The set of continuous functions mapping $\Omega$ into $\real^{n}$,  $k$-times differentiable and with continuous derivatives is denoted $\Ck(\Omega)$. For $p\in \mathcal{C}^1(\Omega)$, the derivative of $p$ with respect to variable $x$ is denoted $\frac{\partial p}{\partial x} = \partial_x p  = p_x$. For $u \in \mathcal{C}^{k}$, $\alpha \in \nt^n_0$, define \begin{small}$$D^{\alpha}u := \left(u_1,\frac{\partial{u_{1}}}{\partial x},  \ldots, \frac{\partial^{\alpha_1}{u_{1}}}{\partial x^{\alpha_1}}, \ldots, \frac{\partial{u_{n}}}{\partial x}, \ldots, \frac{\partial^{\alpha_n}{u_{n}}}{\partial x^{\alpha_n}} \right).$$ \end{small}
Define the \emph{order} of $D^{\alpha}u$  as $ord(D^{\alpha}u) := \max_{j} \alpha_j$. We use $He(\cdot)$ to denote the linear operator $He(A) = A + A^{T}$. For a symmetric matrix $A$  denote $A\geq 0$ ($A>0$) if $A$ is positive definite  (semi-definite). The set of eigenvalues of a matrix $P$ is denoted $\lambda(P)$Elementwise product of two vectors $a$, $b$ is denoted $a \odot b$ while elementwise inequality  is denoted $a \preceq b$.

Consider   $\alpha_\theta= \theta \bf{1}_n$, $\theta \in \nt$, define
\begin{equation} 
v_\theta(u(x)) := D^{\alpha_\theta}u. 
\end{equation} 
The vector $v_\theta$ contains all derivatives of variable $u$ with respect to variable $x$ up to order $\theta$. Variable $u$ is the \emph{dependent variable} and $x\in \Omega \subset \real$ the \emph{independent variable}. 

\section{ Positive functionals and polynomial integrands}
\label{sec:positivefunctionals}

In this paper integral inequalities of the form
\begin{equation}
\int_\Omega \bar{f}(x,v_\theta(u)) dx \geq 0,
\label{ineq:integral}
\end{equation}
are studied, with $\Omega = \left[0,1\right]$. It is assumed that $ \bar{f}(\cdot,v_\theta) \in \mathcal{R}[v_\theta]$, \emph{i.e.} $\bar{f}$ is quadratic on the second argument for any value the first argument assumes, therefore it is possible to write
\begin{equation}
 \bar{f}(x,v_\theta(u))  = v_\theta^{T}(u)F_\alpha (x)v_\theta(u).
\label{eq:gpolynomial}
\end{equation}
It is further assumed that  $F(x) \in \mathcal{C}^{0}(\Omega)$. At the boundary, the dependent variable $u(x)$ takes values satisfying the following linear equation. 
\begin{equation}
B\left[\begin{array}{c} v_{\theta-1}(1) \\v_{\theta-1}(0) \end{array}\right] = 0,
\end{equation}
with $B \in \real^{n_b \times 2 n(\theta-1)}$.

The remaining of this section aims to derive conditions for~\eqref{ineq:integral} to hold in terms of expressions involving only the integrand $\bar{f}(x,v_\theta) $. To this aim, the following result is fundamental

\begin{lemma} \label{lemma:IP}
Consider $r:\Omega  \rightarrow \real^{n_r}$, $r \in \mathcal{C}^{1}$. If there exists a vector function $h:\Omega  \rightarrow \real^{n_r}$, $h \in \mathcal{C}^{1}$ satisfying $h^{T}(x)r(u(x))\leq 0 $ for $x \in \partial \Omega$, then
\begin{equation} \label{eq:IP}
\int_\Omega \left[ h_x (x) r\left( x \right) + h (x) r_x\left( x \right) \right] \,\, dx \leq  0
\end{equation}
\end{lemma}
\begin{proof}
From the fundamental theorem of calculus, one has
\begin{multline*}
h (x)r(x)|_{\partial \Omega} = \int_\Omega \left[ \frac{d}{dx} \left( h (x)r(x)  \right)\right]  \,\, dx  \\ = \int_\Omega \left[ h_x (x) r\left(x \right) + h (x) r_x\left(x \right)  \right]  \,\, dx
\end{multline*}
since $h(x)r(x)\leq 0$ for $x \in \partial \Omega$ one obtains~\eqref{eq:IP}.
\end{proof}

Whenever  $r(x)$ is  a vector of monomials on the elements of $v_\theta(u)$, the integrand in~\eqref{eq:IP} relates the monomials \emph{explicitly} accounting for the dependence of $u$ on variable~$x$ as follows:

\begin{corollary} \label{cor:IP}
Consider  $ v_{\theta-1}^{\{2\}}(u)$, the vector containing all monomials of degree $2$ on $ v_{\theta-1}$, and the set
\begin{multline}
\label{eq:Hset}
\mathcal{H}(k,\theta) \\ :=\left\lbrace  h \in \mathcal{C}^{1}(\Omega) : h(x) \odot  v_{\theta-1}^{\{2\}}(u)|_{\partial \Omega } \preceq  0  \right\rbrace.
\end{multline}
If $h(x) \in \mathcal{H}$, then 
\begin{multline} \label{eq:IPmonomials}
\int_\Omega  \bar{h}(x, v_{\theta}(u))  dx \\ :=  \int_\Omega \left[ h_x(x) \odot v_{\theta-1}^{\{2\}}(u)+ h(x) \odot  C v_{\theta}^{\{2\}}(u)  \right] \,\, dx \preceq  0
\end{multline}
where $C$ is the matrix satisfying $\frac{\partial v_{\theta-1}^{\{2\}}(u)}{\partial x} = C   v_{\theta}^{\{2\}}(u)$.
\end{corollary} 

The corollary is straightforwardly proven by considering $r(x) = v_{\theta}^{\{2\}}$ in~\eqref{eq:IP}. 

\noindent
{\bf Example 1} Consider $\Omega = [0,1]$, $u = u_1$, that is, $n = 1$ and take $\theta = 2$. The set in~\eqref{eq:Hset}, is defined with  $v_{\theta}^{\{2\}} =(u(x)^2, u(x)u_x(x), u_x^2(x))$. Consider $u(0) = u(1) = 0$. The hypothesis of Corollary~\ref{cor:IP} holds with $h(x) = (h_1(x), h_2(x),h_3(x))$ satisfying $h_3(0)\leq 0$ and $h_3(1) \leq 0$ and arbitrary values for $h_1$ and $h_2$ at the boundaries since $u(1)^2 =  u(0)^2 = u(1)u_x(1) = u(0)u_x(0) = 0$. If the values at the boundaries are given by  $u(0) = u(1)$,  $u_x(0) = u_x(1)$, the hypothesis is satisfied with $h_1(1) - h_1(0) \leq 0 $, $h_2(1) - h_2(0) = 0$ and $h_3(1)-h_3(0) \leq 0$. 

\begin{remark}
The parametrization~\eqref{eq:Hset} is defined in terms of the values $h$  assumes at the boundaries of the domain. Thus the integrand $\bar{h}$ in~\eqref{eq:IPmonomials}, which is a vector of $n_r$ elements, can be instrumental to verify~\eqref{ineq:integral}  since if $$\int_\Omega \bar{f}(x,v_\theta(u)) + \sum_i^{n_r} \bar{h}_i(x,v_\theta(u)) dx \geq 0$$ holds then, clearly, $\int_\Omega \bar{f}(x,v_\theta(u)) dx \geq 0$.
\end{remark}

Since~\eqref{eq:gpolynomial} and $\bar{h}(h, v_{\theta}(u)) $ in~\eqref{eq:IPmonomials} are quadratic functions on the dependent variables $v_\theta$, one can write
\begin{equation}
\label{eq:quadform}
\begin{array}{rcl}
\bar{f}(x, v_{\theta}(u)) & = & v_\theta^{T} F(x)  v_\theta\\
 \sum_i^{n_r}  \bar{h}(x, v_{\theta}(u)) & = & v_\theta^{T} H(x)  v_\theta
\end{array}
\end{equation}
with $\bar{k} = \left\lceil \frac{k}{2}\right\rceil$.

{\bf Example 2} Consider $h(x)$ and $v_{\theta}^{\{2\}}$ as in Example~1, then matrix $H(x)$ in~\eqref{eq:quadform} is given by
\begin{multline*}
H(x) =  \left[\begin{array}{ccc}  \partial_x h_{1} &  h_{1} + \frac{1}{2}  \partial_x h_{2} &  \frac{1}{2} h_{2} \\ h_{1}+\frac{1}{2}  \partial_x h_{2} &   h_{2}+   \partial_x h_{3} & h_{3} \\ \frac{1}{2} h_{2} &  h_{3} &0 \end{array}\right].
\end{multline*}

\begin{remark}
Recall that, from the definition of $\mathcal{H}(k,\theta)$, information about the values of the dependent variables at the boundaries define the values at the boundary of the entries of $H(x)$ . 
\end{remark}

\begin{proposition}
\label{prop:integrandSDP}
If  $\exists h \in \mathcal{H}$ (as in~\eqref{eq:Hset}), such that
\begin{equation}
T(x) := F(x) + H(x) \geq 0 \quad \forall x \in \Omega 
\label{eq:Tx}
\end{equation}
with $F(x)$ and $H(x)$ as in~\eqref{eq:quadform}, then inequality~\eqref{ineq:integral} holds.
\end{proposition}
\begin{proof}
Consider $h \in \mathcal{H}$ such that $T(x) \geq 0$ then 
\begin{multline}
0 \leq \int_\Omega  v_\theta^{T} T(x)  v_\theta dx 
\\ 
~~~~~~~\begin{array}{cl}
= &\int_\Omega   v_\theta^{T} \left[ F(x) + H(x) \right] v_\theta dx \\
= &\int_\Omega   v_\theta^{T} F(x)   v_\theta dx  + \int_\Omega   v_\theta^{T}  H(x)  v_\theta dx \\
= & \int_\Omega \bar{f}(x, v_{\theta}(u)) dx + \int_\Omega  \sum_i^{n_r}  \bar{h}_i(x, v_{\theta}(u)) dx \\
\leq &   \int_\Omega \bar{f}(x, v_{\theta}(u))dx
\end{array}
\end{multline}\end{proof}

\begin{remark}
Since the elements of $H(x)$ involve continuously differentiable functions and their derivatives,~\eqref{eq:Tx} is a differential matrix inequality. If we further assume that the functions $h$ and $f$ are polynomials on $x$ it is possible to formulate convex feasibility problem to solve~\eqref{eq:Tx}  as presented in the next section. 
\end{remark}

\section{Positivity in the domain}
\label{sec:domainpositivity}
 
The case of $T(x)$ in~\eqref{eq:Tx} being a polynomial on variable $x$ is addressed in this section. For this class of functions it is possible to formulate the positivity of the matrix in the prescribed domain as a convex optimization problem in the form of SDPs using Positivstellensatz. The following result  is a straightforward application of the Putinar's Positivstellensatz (see Theorem~\ref{thm:Psatz} in the appendix) to~\eqref{eq:Tx}, to hold in the set $\Omega = [0,1]$, characterized as the semi-algebraic set $\{x | x(1-x) \geq 0\}$.

\begin{corollary} \label{cor:psatz}
If there exists $N(x) \in \Sigma^{n_M \times n_M} [x]$ such that
\begin{equation} \label{eq:psatz}
T(x) - N(x) (x)(1-x) \in \Sigma^{n_M \times n_M} [x]
\end{equation}
then~\eqref{eq:Tx} holds.
\end{corollary}

\begin{remark}
If $T(x)$ is affine in the decision variables, which are the parameters $\bar{f}$ and $h$, the above test can be formulated as a SDP whose dimension depends on the degree of the polynomials in variables~$x$.
\end{remark}

\begin{remark}
Although the Positivstellensatz gives necessary and sufficient 
conditions for set containment, in order to make these conditions 
computationally tractable the degree of the sum-of-squares
polynomial $N(x)$ in~\eqref{eq:psatz} must be fixed.
\end{remark}

\section{Stability Analysis for Distributed Parameter Systems}
\label{sec:PDEanalysis}

Consider the following PDE system
\begin{equation} \label{eq:PDEsystem}
u_t = \mathcal{A} u, \quad u(x,0) = u_0(x) \in \mathcal{M} \subset H(\Omega)
\end{equation}
wherein, $H(\Omega)$ is an infinite-dimensional Hilbert space and $\mathcal{A}$ is a linear operator defined on $\mathcal{M}$, a closed subset of $H(\Omega)$.  It is assumed that $\mathcal{A}$ generates a linear semi-group of contractions, i.e., continuous solutions to the PDE exist in $\mathcal{M}$ and are unique. The interested reader can refer to~\cite{CZ95} for details.

In this section we study convergence in $\mathcal{L}_2$-norm of PDEs in one spatial variable and one temporal variable. 

Consider candidate Lyapunov functions of the form
\begin{equation} \label{eq:Lyap}
V(u) = \frac{1}{2} \int_\Omega u^{T}(x)P(x)u(x) dx,~ P(x) > 0~\forall x \in  \Omega
\end{equation}
That is $V(u) =  \frac{1}{2} \|u\|^2_{2,P}$, the squared $P(x)$-weighted $\Ltwo$-norm. Recall that convergence to zero solution and boundedness in a given norm imply convergence and boundedness for an equivalent norm but not for all norms in an infinite dimensional space. The following lemma states the equivalence of the weighted norm and the $\Ltwo$-norm. 
\begin{lemma}
\label{lem:normequivalence}
If $P(x)>0 \quad \forall x \in \bar{\Omega}$ then the norms $\|u\|_{2,P(x)}$ and $\|u\|_{2}$ are equivalent.
\end{lemma}
\begin{proof}
Let $\lambda_M(P,\Omega) := \max_{\bar{\Omega}}(\lambda(P(x)))$, $\lambda_m(P,\Omega) = \min_{\bar{\Omega}}(\lambda(P(x)))$. One has
\begin{multline}
\|u\|_{2,P(x)}^2  = \left[ \int_\Omega u^{T}(x)P(x)u(x) dx \right] \\ \leq \lambda_M(P,\Omega) \left[  \int_\Omega u^{T}(x)u(x) dx \right] = \lambda_M \|u\|_{2}^2 
\end{multline}
\begin{multline}
\|u\|_{2,P(x)}^2  = \left[ \int_\Omega u^{T}(x)P(x)u(x) dx \right] \\ \geq \lambda_m(P,\Omega) \left[   \int_\Omega u^{T}(x)u(x) dx \right] = \lambda_m \|u\|_{2}^2.
\end{multline}
 Therefore
\begin{equation}
\sqrt{\lambda_m(P,\Omega)}  \|u\|_{2}  \leq \|u\|_{2,P(x)} \leq \sqrt{\lambda_M(P,\Omega)}  \|u\|_{2}.\end{equation}\end{proof}

The following proposition is a Lyapunov result for the exponential convergence of the $\Ltwo$ norm of the solutions to~\eqref{eq:PDEsystem} :

\begin{theorem} \label{th1}
Suppose there exists a function $V$ is a functional $V(0) = 0$, and scalars $c_1$, $c_2$, $c_3\in\real_{>0}$ such that
\begin{equation} 
c_1 \|u\|_{2,\Omega} \leq V(u) \leq c_2 \|u\|_{2,\Omega} \label{eq:boundV}
\end{equation} 
\begin{equation} 
V_t(u) \leq  - c_3 \|u\|_{2,\Omega} \label{eq:boundVdot}
\end{equation} 
then the $\Ltwo$ norm of the trajectories of~\eqref{eq:PDEsystem} satisfy
\begin{equation}
\label{eq:expbound}
\|u(t)\|_{2,\Omega} \leq \dfrac{c_2}{c_1} \|u(t_0)\|_{2,\Omega} e^{-\frac{c_3}{c_1}(t-t_0)}
\end{equation}
where $u(t_0) = u(t_0,x)$.
\end{theorem}
\begin{proof}
From \eqref{eq:boundV}-\eqref{eq:boundVdot} one obtains
\begin{equation*}
\dfrac{V_t(u)}{V(u)} \leq -\dfrac{c_3}{c_1}
\end{equation*}
since $\frac{V_t(u)}{V(u)} = \left(\ln(V(u))\right)_t$, the integral of the above expression over $[t_0, t]$, gives
\begin{equation*}
\begin{array}{c}
\int_{[t_0,t]} \left(\ln(V(u(\tau)))\right)_\tau d\tau \leq -\dfrac{c_3}{c_1}(t - t_0)\\
\ln(V(u(t))) -\ln(V(u(t_0))) \leq -\dfrac{c_3}{c_1}(t - t_0) \\
\dfrac{V(u(t))}{V(u(t_0))} \leq e^{-\dfrac{c_3}{c_1}(t - t_0)} \\
V(u(t)) \leq V(u(t_0))e^{-\dfrac{c_3}{c_1}(t - t_0)} 
\end{array}
\end{equation*}
finally \eqref{eq:expbound} is obtained by applying the bounds of~\eqref{eq:boundV} on the above inequality.
\end{proof}

\begin{corollary}
If there exists a function $P(x)$ and positive scalars $\epsilon_1$, $\epsilon_2$ such that
\begin{equation} \label{eq:lyappos}
\frac{1}{2} \int_\Omega u^{T}(x)P(x)u(x) - \epsilon_1 u^{T}(x)u(x)  dx  \geq 0 
\end{equation}
\begin{equation} \label{eq:lyapder}
-  \int_\Omega  u^{T}(x)P(x)\mathcal{A}u(x) + \epsilon_2 u^{T}(x)u(x)  dx \geq 0 
\end{equation}
Then the $\Ltwo$ norm of solutions to~\eqref{eq:PDEsystem} satisfy~\eqref{eq:expbound}.
\end{corollary}

\begin{remark}
Integration-by-parts is a key step to prove stability for PDE systems~\cite{Str04},~\cite{KS08}. It allows to incorporate the boundary conditions when developing the LF time-derivative along the trajectories of the system.\footnote{One example of the application of integration by parts to develop the LF time-derivative is given in the Appendix \ref{app:transport}.} Since the relation among the dependent variables and the boundary conditions are embedded in the polynomials $h(x)$ in~\eqref{eq:IPmonomials}, it is possible to directly treat the derivative condition by studying the integral inequality~\eqref{eq:lyapder}.
\end{remark}

\section{Examples}
\label{sec:examples}
In this section we obtain solutions to the integral inequalities corresponding to Lyapunov stability conditions derived in the previous section. Notice that we can consider $\Omega = [0,1]$ since different one-dimensional domains can be mapped into the unit interval  by means of an appropriate change of variables.

\subsection{Poincar\'e inequality}
\label{ex:poincare}
The Poincar\'e inequality~\cite[p.163]{McO96}
\begin{equation}
\int_{\Omega} u^{2}dx \leq \kappa(\Omega)\int_{\Omega} u_x^{2} dx 
\label{eq:poincare}
\end{equation}
where $\Omega$ is a bounded domain and $\kappa$ is a constant depending on the domain, holds for all $u \in H_0^{1,2}(\Omega)$ and establishes bounds for $\|u\|_2^2$ in terms of $\|u_x\|_2^2$. By rewriting the above inequality as 
\begin{equation}
  \int_{\Omega} \kappa u_x^{2}  - u^{2}  dx \geq 0 
 \label{eq:poincare2}
\end{equation}
one obtains an integral constraint of the form~\eqref{ineq:integral}. Notice that the integrand is affine on~$\kappa$. One may wish to obtain a tight bound for~\eqref{eq:poincare}, \emph{i.e.} find a solution to the following problem 
\begin{equation}
\begin{array}{l}\mbox{minimize } \kappa\\
\mbox{subject to }   \int_{\Omega} \kappa u_x^{2}  - u^{2}  dx \geq 0
\end{array}
\label{optprob:poincare}
\end{equation}

The steps described in Section~\ref{sec:positivefunctionals} are followed by first noticing that the integrand of the integral in \label{eq:poincare2} involves only $u$ and its spatial derivative $u_x$, therefore let $\theta = 1$ in \eqref{eq:IPmonomials} and $v_{\theta-1}(u)=u^2$. Following Proposition~\ref{prop:integrandSDP}  the problem~\eqref{optprob:poincare} becomes
\begin{multline}
\begin{array}{l}\mbox{minimize } \kappa\\
\mbox{subject to } He\left(\dfrac{1}{2} \left[ \begin{array}{cc} -1 + h_x(x) & h(x) \\ 0 &  \kappa  \end{array} \right]\right) \geq 0  
\end{array}\\  \forall x \in \Omega.
\label{optprob:poincare2}
\end{multline}
Assuming $h(x)$ to be of polynomial form, $\Omega = \left[0,1\right]$ and applying  Positivstellensatz as described in Section~\ref{sec:domainpositivity},~\eqref{optprob:poincare2} becomes the following SOSP
\begin{equation}
\begin{array}{l}\mbox{minimize } \kappa\\
\begin{array}{cl} \mbox{subject to } & He\left(\dfrac{1}{2} \left[ \begin{array}{cc} -1 + h_x(x) & h(x) \\ 0 &  \kappa  \end{array} \right]\right) \\& + N(x)x(x-1) \in \Sigma^{2 \times 2}[x], \\ & \quad N(x)  \in \Sigma^{2 \times 2}[x].\end{array}
\end{array}
\label{optprob:poincareSOS}
\end{equation}
The problem~\eqref{optprob:poincareSOS} is formulated and solved using SOSTOOLS considering different degrees for polynomial $h(x)$ and $N(x)$. Figure \ref{fig:poincare} depicts the optimal value $\kappa^{*}$ as a function of the degree of $h(x)$ (the curve was computed setting $deg(N(x)) = deg(h(x))+2$). The figure also presents the optimal bound $\pi^{-2}$ for the studied domain~\cite{PW60}.
\begin{figure}[!htb]
\begin{psfrags}
     \psfrag{degg}[l][l]{\footnotesize $deg(h)$}
     \psfrag{c1}[l][l][1][-90]{\footnotesize $\kappa^{*}$}
\epsfxsize=7cm
\centerline{\epsffile{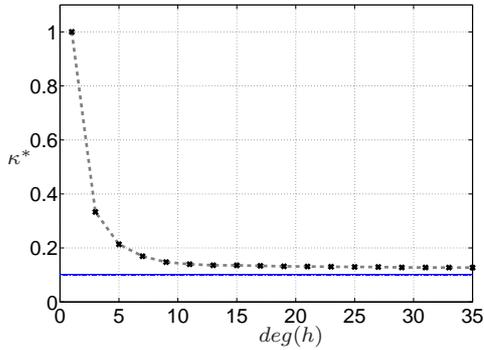}}
\end{psfrags}
\caption{Optimal values for problem~\eqref{optprob:poincareSOS} as a function of the degree of $h(x)$.
\label{fig:poincare}}
\end{figure}

\subsection{The transport equation}
Consider the following PDE
\begin{eqnarray} 
\label{eq:transport}
&u_t = -u_{x} \quad x \in [0,1], \,\, t>0& \nonumber u(0) = 0.
\end{eqnarray}

Let $E_p = \dfrac{1}{2} \int_{\Omega}e^{-\lambda x} u^{2}(x)dx$ be the candidate function to certify $-\lambda E_p - E_{pt} \geq 0$ that is, to certify exponential stability with exponential rate~$\lambda>0$. One has
\begin{equation}
\label{eq:transpg}
-\lambda E_p - E_{pt}  = \int_{\Omega} - \frac{\lambda}{2} e^{-\lambda x} u^2 + e^{-\lambda x}uu_x dx \geq 0,
\end{equation}
which is an inequality as~\eqref{ineq:integral}. Consider $\eta_2(v_1(u)) =  u^2$ and $h(x) = -\frac{1}{2} e^{-\lambda x}$. Since $h(1) = - \frac{1}{2}  e^{-\lambda}<0$, one has $h(1)u^{2}(1) - h(0)u^{2}(0) = h(1)u^{2}(1)<0$, hence $h(x) \in \mathcal{H}(2,1)$ and
\begin{multline*}
h(x)\eta_2(v_1(u))|_{\partial \Omega} =  h(1)u^2(1)    \\ =  \int_{\Omega}  \left( h_xu^2 +h uu_x\right) dx \\= \int_{\Omega}  \frac{1}{2} \lambda e^{-\lambda x}u^2    - e^{-\lambda x} uu_x dx \leq 0
\end{multline*}
where equality holds only if $u(1) = 0$. Adding up $-\lambda E_p - E_{pt}$ and $h(1)u^2(1)$ one obtains
\begin{multline*}
-\lambda E_p - E_{pt} + h(1)u^2(1)  \\ = \int_{\Omega} - \frac{\lambda}{2} e^{-\lambda x} u^2   + e^{-\lambda x}uu_x dx \\+ \int_{\Omega}  \frac{\lambda}{2} e^{-\lambda x} u^2  - e^{-\lambda x} uu_x dx  = 0 
\end{multline*}
therefore
\begin{equation*}
\label{eq:finaltransport}
-\lambda E_p - E_{pt} =  -h(1)u^2(1)  \geq 0, 
\end{equation*}
proving the exponential stability of the zero solution for \emph{any} convergence rate~$\lambda>~0$. This result should be expected as, for a compact and bounded domain, the transport equation is \emph{finite-time} stable. In Appendix~\ref{app:transport} the time-derivative of $E_p$ along the trajectories of~\eqref{eq:transport} is developed with steps using integration by parts to also prove the exponential stability of the zero solution.

By considering inequalities~\eqref{eq:lyappos}-\eqref{eq:lyapder} with a polynomial weighting function and considering polynomial~$h(x)\in \mathcal{H}(2,1)$,  the Positivstellensatz is applied to formulate the following feasibility SOSP
\begin{multline}
\begin{array}{l}
\mbox{find } p(x),~h(x), ~N(x)\\
\mbox{subject to}
\end{array}\\
\begin{array}{l}   He\left(\dfrac{1}{2} \left[ \begin{array}{cc} -\lambda p(x) + h_x(x) & -p(x)+h(x) \\ 0 & 0  \end{array}\right]\right) \\ + N(x)x(x-1) \in \Sigma^{2 \times 2}[x],    \quad N(x)  \in \Sigma^{2 \times 2}[x]. \end{array}
\label{optprob:transport}
\end{multline}

With a polynomial $p(x)$ degree $30$ stability of the zero solution to~\eqref{eq:transport} was certified for $\lambda\in (0, 10]$. The results are depicted in Figure~\ref{fig:transport}.
\begin{figure}[!htb]
\begin{psfrags}
     \psfrag{x1}[l][l]{\footnotesize $x$}
     \psfrag{px}[l][l][1][-90]{\footnotesize $p(x)$}
\epsfxsize=7cm
\centerline{\epsffile{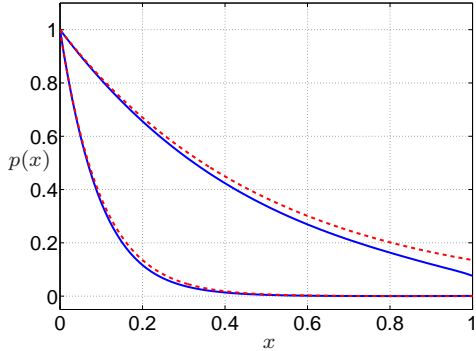}}
\end{psfrags}
\caption{Weighting functions proving exponential stability for convergence rates $\lambda \in \{2,10\}$. The red dotted curves depict the analytical result $\frac{1}{2}e^{-\lambda x}$ while the solid blue lines are correspond to the polynomials obtained by solving~\eqref{optprob:transport}.
\label{fig:transport}}
\end{figure}

\subsection{Heat Equation with Reaction Term}
Consider the following anisotropic PDE
\begin{equation} \label{eq:heatanisotropic}
u_t = u_{xx} + \lambda(x) u,~x \in [0,1],~ u(0)= u(1) = 0 
\end{equation}
where, $\lambda : [0,1] \to \mathbb{R}$. When $\lambda(x) = \lambda_c$, the Lyapunov functional $\int_{0}^1 u^2 \,\, dx$, proves asymptotic stability for $\lambda_c \in (-\infty,\pi^2)$ (see Appendix \ref{app2}). In order to study the exponential stability of~\eqref{eq:heatanisotropic}  consider a weighted $\Ltwo$ function as~\eqref{eq:Lyap}. 

In~\cite{PP06} the system was studied with $\lambda(x) = \lambda_c$ and employing an  \emph{ad hoc} integration by parts construction to obtain a tight estimate for the stability interval. Here $\lambda(x)$ is considered as $\lambda(x) = \lambda_c - 24x + 24x^2$ and a line search was performed maximize the coefficient $\lambda_c$ for which the system is stable. We obtained the value $ \lambda_c^{*} = 14.1$ by solving~\eqref{eq:lyappos}-\eqref{eq:lyapder} with a polynomial weighting function.  Figure~\ref{fig:heatcoeff}  depicts $\lambda(x)$ with the obtained value.  The stability bound for a constant coefficient $\lambda$, $\pi^2$, is also depicted. Notice that for some $\lambda(x)> \pi^2$ for some values of $x$. The obtained weighting function $p(x)$, a polynomial of degree $10$ is illustrated in Figure~\ref{fig:heatcoeffLF}

\begin{figure}[!htb]
\begin{psfrags}
     \psfrag{x}[l][l]{\footnotesize $x$}
     \psfrag{lamx}[c][c][1][-90]{\footnotesize $\lambda(x)$}
\epsfxsize=6cm
\centerline{\epsffile{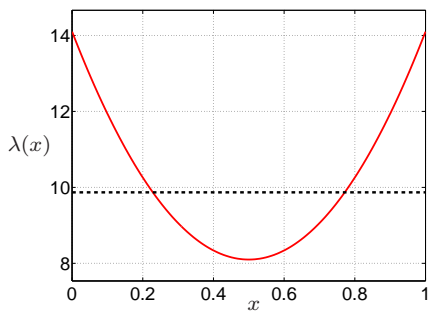}}
\end{psfrags}
\caption{The spatially varying coefficients $\lambda = \pi^2$ (dashed black) $\lambda(x) = \lambda_c - 24x + 24x^2$ (solid red).
\label{fig:heatcoeff}}
\end{figure} 

\begin{figure}[!htb]
\begin{psfrags}
     \psfrag{x}[l][l]{\footnotesize $x$}
     \psfrag{px}[c][c][1][-90]{\footnotesize $p(x)$}
\epsfxsize=6cm
\centerline{\epsffile{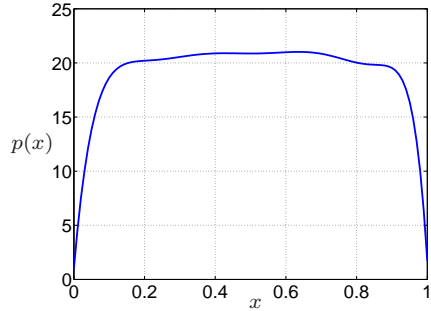}}
\end{psfrags}
\caption{The weighting function $p(x)$ for system~\eqref{eq:heatanisotropic}.
\label{fig:heatcoeffLF}}
\end{figure}

%
%
%

\subsection{System of PDEs coupled via reaction term} \label{ss1}
Consider the following system of PDEs inspired by~\cite[p 38]{Str04}
\begin{align}  \label{eq:coupled}
\left\lbrace\begin{array}{rcl}
u_t &=& \frac{1}{R} u_{xx} + \alpha u +\gamma v \\
 v_t &=& \frac{1}{R} v_{xx}   +\delta u + \beta v
\end{array}\right.,\\x \in [0,1], ~u(0) = u(1) =v(0) = v(1) = 0 \nonumber
\end{align}
where, $\alpha = 1$, $\gamma = 1.5$, $\delta =5$ and $\beta = 0.2$. Through simulation it is observed that for $R < 2.7$ trajectories converge to the zero solution.   

We consider the energy and functionals~\eqref{eq:Lyap} of different degrees.  The results are depicted in Table~\ref{table:coupled}. Figure~\ref{fig:coupled} details the solution for $deg(P(x)) = 4$, $P = \left[\begin{smallmatrix} P_{11}&P_{12}\\ P_{12}&P_{22} \end{smallmatrix}\right]$ showing the values of the entries of the weighting matrices and its eigenvalues. 

\begin{table}[!h]
\caption{Stability intervals for parameter $R\in (0,R^{*}]$ for different degrees of $P(x)$.\label{table:coupled}}
\centering
\begin{tabular}{|c||c|c|c|c|c|c|}
\hline 
\bfseries $deg(P(x))$ & 0 ($P(x) = I$) & 0 & 2& 4 & 6& 8\\
\hline
\bfseries $R^{*}$  & 0.2 &  0.3 &  1.7 & 2.3 & 2.4 & 2.45 \\
\hline
\end{tabular}
\end{table}

\begin{figure}[ht]
\begin{psfrags}
     \psfrag{x}[l][l]{\footnotesize $x$}
     \psfrag{Px}[c][c][1][-90]{\footnotesize $\lambda(P(x))~~$}
\epsfxsize=6cm
\centerline{\epsffile{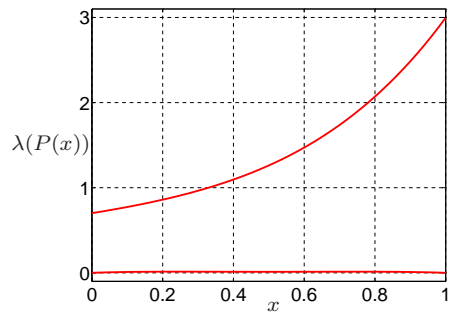}}
\end{psfrags}
\caption{Eigenvalues of $P(x) = $ of degree~$4$. Notice that both eigenvalues are positive.
\label{fig:coupled}}
\end{figure}
 

%

\section{Conclusion}\label{se:conclusion}

This paper has formulated conditions for the positivity of functional inequalities in terms of positivity of their integrands by characterizing a set of expressions constructed from the Fundamental Theorem of Calculus. The main assumption is that the functionals under study are polynomial on the dependent variables.  The case of polynomial dependence of the integrand on the independent variable allows for the formulation of a convex optimization problem given by SDPs. 

These formulations were then used to study integral inequalities arising from Lyapunov stability conditions for PDEs. Several examples illustrate the effectiveness of the proposed approach. The examples are instances of the set of PDEs which are polynomial on the dependent variable and its derivatives. 

Polynomial parametrization of the weighting functions on the Lyapunov functionals is not restrictive since, according to Weierstrass approximation theorem, any continuous function on a bounded interval can be approximated by a polynomial.   The drawback is that the degree of the approximating polynomial may not be known \emph{a priori}.

The research leading to the results  presented here was motivated from the fact that integration by parts is a crucial step on the stability analysis. The local checks, which are often provided by embedding theorems on bounded domains, are also important. Our scope was to make these steps computationally tractable by formulating SDPs. However, we believe the results presented in sections~\ref{sec:positivefunctionals} and~\ref{sec:domainpositivity} go beyond the scope of stability analysis of PDEs, providing an efficient method of formulating a set of optimization problems with integral constraints in a convex optimization framework.

\bibliography{references}
\bibliographystyle{IEEEtran}


\appendix

\subsection{Sum-of-Squares Polynomials}
A polynomial $p(x)$ is a sum-of-squares polynomial if $\exists p_i(x) \in \mathcal{R}[x]$, $i \in \{1, \ldots, n_d\}$ such that $p(x) = \sum_i p_i^2(x)$. Hence $p(x)$ is clearly non-negative. A set of polynomials $p_i$ is called \emph{SOS decomposition} of $p(x)$. The converse does not hold in general, that is, there exist non-negative polynomials which do not have an SOS decomposition~\cite{Par00}.  The computation of SOS decompositions, can be cast as an SDP (see~\cite{CLR95,Par00,CTVG99}). The Theorem below proves that, in sets satisfying a property stronger than compactness, any positive polynomial can be expressed as a combination of sum-of-squares polynomials and polynomials describing the set.  

For a set of polynomials $\bar{g} = \{g_1(x), \ldots, g_m(x)\}$, $m \in \nt$, the \emph{quadratic module} generated by $m$ is 
\begin{equation}
M(\bar{g}):= \left\lbrace \sigma_0 +\sum_{i = 1}^{m} \sigma_i g_i | \sigma_i \in \Sigma[x]\right\rbrace.
\end{equation}
A quadratic module $M\in \mathcal{R}[x]$ is said \emph{archimedean} if $\exists N \in \nt $ such that $$N - \|x\|_2^2 \in M.$$
An archimedian set is always compact~\cite{NS08}. It is the possible to state~\cite[Theorem 2.14]{Las09}
\begin{theorem}[Putinar Positivstellensatz]
\label{thm:Psatz}
Suppose the quadratic module $M(\bar{g})$ is archimedian. Then for every $f \in \mathcal{R}[x]$, $$f>0~\forall~x\in \{x | g_1(x)\geq 0, \ldots, g_m(x)\geq 0 \} \Rightarrow f \in (\bar{g}).$$
\end{theorem}

\begin{lemma}The set $\Omega = [0,1]$ is Archimedean.
\end{lemma}
Take any pair $(r,N^{*})$, $r\in \real_{>0}$ and $N^{*}\in \mathbb{N}$ satisfying 
\begin{equation}
N^{*} \geq \frac{1}{4}\frac{r^2}{r-1}.
\end{equation}
The Archimedean property is the satisfied with 
\begin{equation*}
\begin{array}{rcl}
\theta_0(\sigma) &= & \left(\left( \sqrt{r-1}\right)\sigma - \frac{1}{2}\frac{r}{\sqrt{r-1}} \right)^2\\
 &&  + \left( N^{*}-\frac{1}{4} \frac{r^2}{(r-1)}\right)\\
\theta_1(\sigma) &= &r.
\end{array}\\
\end{equation*}

\subsection{Stability Bounds for the Heat Equation with Reaction Term} \label{app2}
The stability bound on parameter $\lambda$ is obtained by constructing the solution to
\begin{equation} \label{eq:heat_reac}
u_t = u_{xx} + \lambda u, \quad \forall x \in [0,1]  \quad u(0)=u(1) =0
\end{equation}
Assuming separation of variables for the solution, a candidate solution can be written as
\begin{equation} \label{eq:solsep}
u(x,t) = X(x)T(t).
\end{equation}
Substituting~\eqref{eq:solsep} in~\eqref{eq:heat_reac}, one obtains $T_t X = X_{xx}T + \lambda XT$,  that is,
\begin{equation} \label{eq:solsep2}
\frac{T_t}{T} = \frac{X_{xx}+\lambda X}{X}.
\end{equation}
The left hand side of~\eqref{eq:solsep2} is only a function of $t$, and the right hand side, a function of $x$. Consequently,  
\begin{equation} \label{eq:solsep3}
\frac{T_t}{T} = \frac{X_{xx}+\lambda X}{X} = k
\end{equation}
for some constant $k$. It can be verified, using the boundary conditions, that the parameter $k$ should be positive for~\eqref{eq:heat_reac} to have a non-trivial solution, yielding
\begin{equation}
X_{xx} + (\lambda - k)X = 0,
\end{equation}
of which the solution is $X(x) = c_1 \sin (\sqrt{\lambda - k}x) + c_2 \cos (\sqrt{\lambda-k}x)$. Employing the boundary conditions of~\eqref{eq:heat_reac}, one obtains $c_2 = 0$ and
\begin{multline}
c_1 \sin (\sqrt{\lambda - k}) = 0 \Rightarrow \sqrt{\lambda - k} = n\pi  \\ \Rightarrow k = \lambda - n^2 \pi^2, \quad n \in \nt.
\end{multline}
Then, it follows that from~\eqref{eq:solsep3} one has
$$
\frac{T_t}{T} = \lambda - n^2 \pi^2 \Rightarrow T(t) = e^{-(n^2\pi^2 - \lambda)t}, \quad n \in \nt.
$$
Therefore, for the system to be stable, the following must hold
$$
n^2\pi^2 - \lambda > 0, \quad n  \in \nt,
$$
that is, $\lambda < \pi^2$.

\subsection{Lyapunov function for the transport equation}
\label{app:transport}
Consider the system
\begin{equation}
u_t = - u_{x} \quad \Omega = \left(0,1\right) \quad u(0) = 0,
\end{equation}
and the candidate Lyapunov function  of the form $$E_p = \dfrac{1}{2} \int_{\Omega}e^{-\lambda x} u^{2}(x)dx.$$
One obtains
\begin{equation}
\begin{array}{rcl}
E_{pt} & = &   \int_{\Omega}e^{-\lambda x} uu_tdx \\
& = &  - \int_{\Omega}e^{-\lambda x} uu_xdx \\
& = &  -  \frac{1}{2} \int_{\Omega} \left[ \frac{d}{dx}\left(e^{-\lambda x} u^{2}\right)  + \lambda  e^{-\lambda x} u^{2} \right] dx \\
& = &  -  \frac{1}{2} \int_{\Omega}  \frac{d}{dx}\left(e^{-\lambda x} u^{2}\right) dx - \frac{\lambda}{2}  \int_{\Omega} e^{-\lambda x} u^{2} dx\\
& = &  -  \frac{1}{2} \left[ e^{-\lambda} u^{2}(1) - u^{2}(0)   \right]- \frac{\lambda}{2}  \int_{\Omega} e^{-\lambda x} u^{2} dx\\
& = &  -  \frac{1}{2}  e^{-\lambda} u^{2}(1)   - \frac{\lambda}{2}  \int_{\Omega} e^{-\lambda x} u^{2} dx\\
& \leq &   - \frac{\lambda}{2}  \int_{\Omega} e^{-\lambda x} u^{2} dx.
\end{array}
\end{equation}
That is, $E_{pt}  \leq -\lambda E_p$, which proves the exponential stability of the zero solution.

%


\end{document}